\documentclass[a4paper,11pt]{amsart}

\usepackage[utf8]{inputenc}
\usepackage[english]{babel}
\usepackage{amsfonts}
\usepackage{amsmath}
\usepackage{mathrsfs}
\usepackage{amssymb}
\usepackage{amsthm}
\usepackage{amscd}
\usepackage{latexsym}
\usepackage{amstext}
\usepackage{amsxtra}
\usepackage[alphabetic,backrefs,lite]{amsrefs} 
\usepackage{color}
\usepackage{paralist}
\usepackage[colorinlistoftodos]{todonotes}
\usepackage[all]{xy}

\usepackage{enumerate}
\usepackage{multirow}
\usepackage{wasysym}
\usepackage{latexsym} 
\usepackage{comment}
\usepackage{colonequals}

\usepackage[colorinlistoftodos]{todonotes}

\usepackage{geometry}
\usepackage{hyperref}
\hypersetup{
        citecolor=blue,
        draft=false,
        colorlinks, 
        pdfauthor={D. Bricalli, F. F. Favale},
        pdftitle={Lefschetz properties for jacobian rings of cubic fourfolds and other Artinian algebras},
        linktocpage
}

\newtheorem{theorem}{Theorem}[section]
\newtheorem*{theorem*}{Theorem}
\newtheorem{lemma}[theorem]{Lemma}
\newtheorem{corollary}[theorem]{Corollary}
\newtheorem{proposition}[theorem]{Proposition}
\newtheorem{remark}[theorem]{Remark}
\newtheorem{construction}[theorem]{Construction}
\newtheorem{definition}[theorem]{Definition}
\newtheorem{example}[theorem]{Example}

\newcommand{\nc}{\newcommand} 
\nc{\cH}{{\mathcal H}}
\nc{\cA}{{\mathcal A}}
\nc{\cG}{{\mathcal G}}
\nc{\cC}{{\mathcal C}}
\nc{\cO}{{\mathcal O}}
\nc{\cI}{{\mathcal I}}
\nc{\cB}{{\mathcal B}}
\nc{\cY}{{\mathcal Y}}
\nc{\cK}{{\mathcal K}} 
\nc{\cX}{{\mathcal X}}
\nc{\cS}{{\mathcal S}}
\nc{\cE}{{\mathcal E}}
\nc{\cF}{{\mathcal F}}
\nc{\cZ}{{\mathcal Z}}
\nc{\cQ}{{\mathcal Q}}
\nc{\cN}{{\mathcal N}}
\nc{\cP}{{\mathcal P}}
\nc{\cL}{{\mathcal L}}
\nc{\cM}{{\mathcal M}}
\nc{\cT}{{\mathcal T}}
\nc{\cW}{{\mathcal W}}
\nc{\cU}{{\mathcal U}}
\nc{\cJ}{{\mathcal J}}
\nc{\cV}{{\mathcal V}}
\nc{\bH}{{\mathbb H}}
\nc{\bA}{{\mathbb A}}
\nc{\bG}{{\mathbb G}}
\nc{\bC}{{\mathbb C}}
\nc{\bO}{{\mathbb O}}
\nc{\bI}{{\mathbb I}}
\nc{\bB}{{\mathbb B}}
\nc{\bY}{{\mathbb Y}}
\nc{\bK}{{\mathbb K}} 
\nc{\bX}{{\mathbb X}}
\nc{\bS}{{\mathbb S}}
\nc{\bE}{{\mathbb E}}
\nc{\bF}{{\mathbb F}}
\nc{\bZ}{{\mathbb Z}}
\nc{\bQ}{{\mathbb Q}}
\nc{\bN}{{\mathbb N}}
\nc{\bP}{{\mathbb P}}
\nc{\bL}{{\mathbb L}}
\nc{\bM}{{\mathbb M}}
\nc{\bT}{{\mathbb T}}
\nc{\bW}{{\mathbb W}}
\nc{\bU}{{\mathbb U}}
\nc{\bD}{{\mathbb D}}
\nc{\bJ}{{\mathbb J}}
\nc{\bV}{{\mathbb V}}
\nc{\bbZ}{{\mathbb Z}}
\nc{\bR}{{\mathbb R}}
\nc{\fr}{{\rightarrow}}
\nc{\co}{{\nabla}}

\nc{\cu}{{\barline{\nabla}}}

\usepackage{mathtools}

\DeclareMathOperator{\Sec}{Sec}

\DeclareMathOperator{\Sing}{Sing}
\DeclareMathOperator{\Supp}{Supp}

\title[Lefschetz properties for cubic fourfolds and other Artinian algebras]{Lefschetz properties for jacobian rings of cubic fourfolds and other Artinian algebras}

\author{Davide Bricalli}
\address{Dipartimento di Matematica e Applicazioni,
	Universit\`a degli Studi di Milano-Bicocca,
	Via Roberto Cozzi, 55,
	I-20125 Milano, Italy}
\email{d.bricalli1@campus.unimib.it}

\author{Filippo Francesco Favale}
\address{Dipartimento di Matematica,
	Universit\`a degli Studi di Pavia,
	Via Ferrata, 5
	I-27100 Pavia, Italy}
\email{filippo.favale@unipv.it}

\date{\today}

\thanks{
\textit{2020 Mathematics Subject Classification}: Primary: 14J35; Secondary: 13E10, 14J70, 14M10, 13F20, 14N05\\
\textit{Keywords}: Artinian Gorenstein algebras, Cubic fourfolds, Jacobian rings, Lefschetz properties, complete intersections}

\begin{document}

\maketitle

\begin{abstract}
    In this paper, we exploit some geometric-differential techniques to prove the strong Lefschetz property in degree $1$ for a complete intersection standard Artinian Gorenstein algebra of codimension $6$ presented by quadrics. We prove also some strong Lefschetz properties for the same kind of Artinian algebras in higher codimensions. Moreover, we analyze some loci that come naturally into the picture of "special" Artinian algebras: for them we give some geometric descriptions and show a connection between the non emptiness of the so-called non-Lefschetz locus in degree $1$ and the "lifting" of a weak Lefschetz property to an algebra from one of its quotients. 
\end{abstract}



\section*{Introduction}

In this article we exploit and generalize some techniques that the authors have developed in previous works on these topics, in collaboration with Gian Pietro Pirola (see \cite{FP21} and \cite{BFP}) in order to prove the validity of some Lefschetz properties for particular Artinian algebras. More precisely, we deal with standard Artinian Gorenstein graded algebras over a field $\bK$ of characteristic $0$. For brevity, we refer to such an algebra using the acronym {\it SAGA}. SAGAs are generated in degree $1$ and satisfy a sort of "inner duality" property induced by the products of elements in complementary degrees (see Section \ref{DEF:ARTGOR} for details). We recall that the codimension of a standard Artinian graded algebra $R$ is the dimension of $R^1$, i.e. its degree-1 part, as vector space over $\bK$.

Classical examples of these algebras are jacobian rings of smooth hypersurfaces in projective spaces. These are constructed by considering the smooth hypersurface $X=V(F)\subseteq \bP^n$ with $F$ homogeneous in $S=\bK[x_0,\cdots,x_n]$ and by taking the quotient $R=S/I$ where $I$ is the ideal generated by the partial derivatives of $F$ (the so-called jacobian ideal of $F$). The importance of these objects lies, for example, in the strong geometric relations between $X$ and its jacobian ring. Just to mention some, in the seminal works \cite{CG,IVHS_I,IVHS_II,IVHS_III}, Carlson, Griffiths, Green and Harris proved that a portion of the primitive part of the Dolbeault cohomology is codified in $R$ and that $R$ plays a crucial role in the study of the infinitesimal variation of Hodge structure of $X$.

This construction is a special case of a more general one. Indeed, one can also consider the quotient $S/I$, where $I$ is generated by a regular sequence of length $n+1$ of homogeneous polynomials, i.e. $I=(f_0,\cdots,f_n)\subset\bK[x_0,\cdots,x_n]$, where $V(f_0,\cdots,f_n)\subset\bP^n$ is the empty set. This kind of algebras are SAGAs and are also complete intersection rings so it makes sense to refer to them as {\it complete intersection SAGAs}. For brevity one often says (see, for example, \cite{MN13}) that $R$ is a {\it complete intersection SAGA presented by forms of degree $e$}, if the generators of the ideal $I$ have all the same degree, equal to $e$. Our main result (Theorem \ref{THM:SLP1}) will deal with the case $e=2$, i.e. with complete intersection SAGAs presented by forms of degree $2$.
\medskip

Lefschetz properties for an Artinian algebra were defined in the 80's by taking inspiration from the Hard Lefschetz Theorem for a K\"ahler variety $X$ of complex dimension $n$. We recall that this theorem states that for all $k\leq n$ the cup product map $\omega^{k}\cdot :H^{n-k}(X,\bC)\to H^{n+k}(X,\bC)$ is an isomorphism for $\omega\in H^2(X,\bC)$ generic (more precisely, for any K\"ahler form in $H^{1,1}_{\bar{\partial}}(X)$). Then, the even cohomology ring $R=\bigoplus_{i=0}^{n}R^i=\bigoplus_{i=0}^{n}H^{2i}(X,\bC)$ of $X$ is a Gorenstein Artinian algebra. In particular, $R$ is such that for $x\in R^1$ general, the multiplication map $x^{n-2i}\cdot: R^i\to R^{n-i}$ is an isomorphism for $i\leq n/2$. This is, roughly, the definition of the {\it strong Lefschetz property} ($SLP$) for an Artinian algebra (for details see Definition \ref{DEF:LefProp}). Starting from this idea a weaker version has been introduced by following similar ideas: an algebra $R$ has the {\it weak Lefschetz property} ($WLP$) if the multiplication map $x\cdot :R^k\to R^{k+1}$ is of maximal rank for all $k\geq 0$ and $x\in R^1$ general.
For a comprehensive treatment of Lefschetz properties the interested reader can refer to \cite{bookLef}.
\medskip

In the recent years, there has been a growing interest in Artinian algebras and their Lefschetz properties. We summarize here some results and conjectures relevant with respect to the topics of our article. In \cite{HMNW} it is shown that any Artinian standard algebra $R$ of codimension $2$ satisfies $SLP$. In the same article, it is also proved that if $R$ has codimension $3$ and is a complete intersection then it satisfies $WLP$. Again in codimension $3$, if one restricts its attention to jacobian rings of curves of degree $d$ in $\bP^2$, the validity of $SLP$ is known only up to degree $4$ with the case for $d=4$ treated in the very recent work \cite{DGI20}. In codimension $4$, again in \cite{DGI20}, the authors prove that jacobian rings of cubic surfaces have the $SLP$. In particular, nothing is known for Lefschetz properties for jacobian rings when the degree increases: quintic curves in $\bP^2$ (regarding $SLP$) and quartic surfaces in $\bP^3$ (regarding both $WLP$ and $SLP$) are the first open cases.
For what concerns the case of complete intersection SAGAs presented by quadrics, in \cite{MN13} it is shown that $WLP$ holds in degree $1$ (i.e. $x\cdot :R^1\to R^2$ is injective for $x\in R^1$ general) in any codimension. If one wonders if all SAGAs have to satisfy $WLP$, this is known to be false as there are counterexamples (which are not complete intersection) in higher codimension (\cite{GZ18}) but it is conjectured to hold for complete intersection SAGAs. For codimension bigger than or equal to $5$ very little is known. For codimension $5$, $WLP$ has been proved in \cite{AR19} for complete intersection SAGAs presented by quadrics whereas, in \cite{BFP} the authors proved $SLP$ for the same SAGAs (in particular, for jacobian rings of smooth cubic threefolds). The cases of higher codimensions or higher degrees are, at the moment, completely open. 

Just to mention some other papers, which treat problems related to Lefschetz properties for Artinian Algebras, we cite, for example, \cite{MMN, MMO, Ila18, AAISY} and \cite{DI22}. These articles are really interesting on their own since the approaches used to study, roughly, the "same" problems are really different and involve a lot of different techniques.

In this work, we follow the "geometric-differential" approach used in \cite{AR19} and \cite{BFP} by exploiting and generalising the techniques of the latter, in order to prove, partially, the $SLP$ for the first open case with low degree in higher codimension. More precisely, we deal with a complete intersection SAGA $R$ presented by quadrics of codimension $6$ and we prove our first main result:
\begin{theorem*}{(Theorem \ref{THM:SLP1})}
Let $R$ be a SAGA as above. Then $R$ satisfies the $SLP$ in degree $1$, i.e. for $x$ general in $R^1$ the multiplication map $x^4\cdot:R^1\rightarrow R^5$ is a bijection.
\end{theorem*}
This result is a step towards the evidence of a well known conjecture, according to which, in characteristic $0$ all complete intersection SAGAs should satisfy the Lefschetz properties (see \cite[Conjecture 3.46]{bookLef}, for example). Beyond being the first open case, the interest for these SAGAs lies in the fact that they involve Lefschetz properties for the jacobian rings of cubic fourfolds, which are very interesting on their own for what concerns several other unrelated topics (such as, for example, the well known Kuznetsov's conjecture about the rationality of the cubic fourfolds).
\medskip

We stress that Lefschetz properties, at least for a {\it general} complete intersection SAGA $S/I$ with fixed degrees of the generators of $I$, are known to hold. For this reason, it is interesting to analyze "special" SAGAs. With this idea, we proceed by obtaining two interesting results concerning special SAGAs. The first one is a collection of geometric properties satisfied by the subschemes 
$$\cN_i:=\{[x]\in \bP(R^1)\,|\, x^i=0\}$$
that we call {\it Nihilpotent loci}.  
Roughly, when $R$ is "general", $\cN_i$ is empty for low values of $i$, then, in case of non-emptiness, we do expect that these loci reflect some properties of "special" SAGAs. The most relevant results are Proposition \ref{PROP:Y2points}, which has as a consequence, a characterization of Fermat hypersurfaces (see \ref{COR:ReconstructFermat}) and Theorem \ref{THM:SEC} which analyzes the geometric behaviour of these loci when they contain a linear space.
\medskip

For the second result, we consider what are called the {\it non-Lefschetz loci}, i.e. the subschemes parametrising linear forms for which the injectivity of the multiplication map fails (see Definition \ref{DEF:LefLocus}). These are really interesting for the geometry and the algebraic structure of a SAGA. Just to mention some articles which study these loci, one can consider \cite{BMMN} and \cite{AR19}. Our contribution for this topic is the following lifting criterion:
\begin{theorem*}{(Theorem \ref{THM:induction})}
Let $R=S/I$ be a complete intersection SAGA of codimension $n+1\geq 6$ presented by quadrics and assume that $z$ is a non-Lefschetz element for $R$. If $\bar{R}=R/(z)$ satisfies the $WLP$ in degree $2$, then the same holds for $R$.
\end{theorem*}
This is a result for "special" SAGAs since, as for the Nilpotent loci, when $R$ is general, all $[x]\in \bP(R^1)$ are Lefschetz elements. This unexpected theorem gives a sort of converse for results which prove that Lefschetz properties are inherited by suitable quotients. This happens, for example, for $SLP$ when one takes the quotient by the conductor of a Lefschetz element (\cite[Proposition 3.11]{bookLef}) or, for the $WLP$ as shown in the recent article \cite{Gue}.
\medskip

The plan of the article is the following. In Section \ref{SEC:Notations}, we set our notation and we introduce the main characters of the paper. Moreover, we establish a general framework that will be used in several specific situations through the paper, and we prove the first general technical results. 
In Section \ref{SEC:Nihil}, we analyze the nihilpotent loci $\cN_i$ for a SAGA $R$ regardless of the validity of some Lefschetz properties. We study them from a geometric point of view and we recollect some examples which show how the nature of these loci can be very different from one case to another. 
In Section \ref{SEC:SLP1} we prove the first main result, i.e. Theorem \ref{THM:SLP1}.
In Section \ref{SEC:WLP}, we analyze the weak Lefschetz property for complete intersection SAGAs presented by quadrics and prove the "lifting" of the $WLP$ to $R$ from suitable quotients, i.e. Theorem \ref{THM:induction}.
Finally, one might wonder how much the different techniques developed in this paper could be exploited to study the next open cases. In Section \ref{SEC:Higher} we try to answer to this question, by generalizing the validity of some strong Lefschetz properties to complete intersection SAGAs presented by quadrics of higher codimension.
\medskip

\noindent {\bf Acknowledgements}: \\
The authors are thankful to Gian Pietro Pirola for his constant support and his priceless suggestions and to Alberto Alzati and Riccardo Re for the helpful discussions on the topics of the paper. The authors are partially supported by INdAM - GNSAGA and by PRIN \emph{``Moduli spaces and Lie theory''} and by (MIUR): Dipartimenti di Eccellenza Program (2018-2022) - Dept. of Math. Univ. of Pavia.\\


\section{Notations, preliminaries and technical results}
\label{SEC:Notations}

Through all the article we will assume that $\bK$ is an algebraically closed field of characteristic $0$. 
First of all, let us recall two basic definitions that will play an important role in what follows (the reader can refer to \cite{bookLef} for a deeper treatment). 

\begin{definition}
\label{DEF:ARTGOR}
Consider an Artinian graded $\bK$-algebra $R=\bigoplus_{i=0}^NR^i$. Then
\begin{itemize}
    \item the dimension of $R^1$ as $\bK$-vector space is called {\bf codimension} of $R$;
    \item $R$ is {\bf standard} if it is generated by $R^1$ as $\bK$-algebra;
    \item $R$ is said to have the {\bf Gorenstein duality} if $R^N\simeq \bK$ and the multiplication map $R^j\times R^{N-j}\to R^{N}$ is a perfect pairing whenever $0\leq j\leq N$.
\end{itemize}
If $R$ is a graded Artinian algebra, having the Gorenstein duality is equivalent to ask that $R$ is Gorenstein. In this case, $R^N$ is called {\bf socle} of $R$ and $R^0\simeq R^N$. A standard Artinian Gorenstein algebra will be called, for brevity, {\bf SAGA}. 
\end{definition}

Examples of SAGAs are algebras that can be written as $\bK[x_0,\dots, x_n]/I$ with $I$ an homogeneous ideal generated by a regular sequence of polynomials of length $n+1$. We will be interested in the case where the generators of $I$ have all the same degree $e$: for brevity, in this situation we say that $R$ is a {\bf complete intersection SAGA presented by forms of degree $e$}. Relevant examples of this kind of SAGAs are jacobian rings of smooth hypersurfaces of $\bP^n$.
\medskip

We now introduce Lefschetz properties for a SAGA $R$. These involve the multiplication maps by suitable elements of $R$. In the following we will often deal with kernels of such maps so it is convenient to set 
$$K^a_q:=\ker\left(q\cdot :R^a\to R^{a+s}\right)\qquad \mbox{ for }q \in R^s.$$ 

\begin{definition}
\label{DEF:LefProp}
Consider a SAGA $R=\bigoplus_{i=0}^NR^i$. It is said to satisfy the 
\begin{itemize}
    \item weak Lefschetz property in degree $k$, ($WLP_k$ in short) if there exists $x\in R^1$ such that $x\cdot :R^k\to R^{k+1}$ has maximal rank;
    \item strong Lefschetz property in degree $k$ at range $s$, ($SLP_k(s)$ in short) if there exists $x\in R^1$ such that $x^{s}\cdot :R^k\to R^{k+s}$ has maximal rank.
\end{itemize}
The algebra $R$ has the strong Lefschetz property in degree $k$ ($SLP_k$) if $SLP_k(s)$ holds for all $s$. We also say that $R$ satisfies {\bf weak (strong) Lefschetz property} - $WLP$ (respectively $SLP$) in short - if it satisfies $WLP_k$ (respectively $SLP_k$) for all $k$. 
\end{definition}

\begin{remark}
For SAGAs, the above definition of $SLP$ is equivalent to ask $SLP_k(N-2k)$ for all $k\leq N/2$ (see Definition $3.18$ and subsequent discussion in \cite{bookLef}). Moreover, notice that if $k\leq N/2$ and $1\leq s\leq N-2k$, $SLP_k(s)$ implies $SLP_k(s-1)$. Note that, by definition, $WLP_k$ is equivalent to $SLP_k(1)$ and, by Gorenstein duality, $WLP_k$ holds if and only if $WLP_{N-k-1}$ holds. Finally, for $k\leq N/2$, $WLP_k$ implies $WLP_{k-1}$ by \cite[Proposition 2.1]{MMN}.
\end{remark}

Elements of $\bP(R^1)$ for which the multiplication map is not of maximal rank will be important in Section \ref{SEC:WLP} so it is convenient to introduce the following subschemes of $\bP(R^1)$.

\begin{definition}
\label{DEF:LefLocus}
Let $R$ be any SAGA with socle in degree $N$. For $1\leq a\leq N-1$ we define the {\bf Lefschetz locus in degree $a$} to be
$$\cL_a:=\{[x]\in \bP(R^1)\,|\, x\cdot :R^a\to R^{a+1} \mbox{ has maximal rank}\}\subset \bP(R^1).$$
An element $[x]\in \bP(R^1)$ (or, equivalently, $x\in R^1\setminus\{0\}$) is called {\bf Lefschetz element in degree $a$} if $[x]\in \cL_a$. On the contrary, elements not in $\cL_a$ are called {\bf non-Lefschetz elements (in degree $a$)}. 
\end{definition}

Geometric results on these loci can be found, for example, in \cite{AR19} and \cite{BMMN}. 
\medskip

In the next sections, we will make a large use of the framework that we present here in a general setting. This kind of construction has been used in some articles (see, for example \cite{BFP} and \cite{AR19}) in order to study from a geometric perspective SAGAs which do not satisfy some Lefschetz properties.

Let $R$ be a SAGA of codimension $n+1$ and socle in degree $N$. 
For $1\leq a,b\leq N$ we set
$$\Gamma_{i,j}^{(a,b)}=\{([x],[y])\in \bP(R^a)\times \bP(R^b)\,|\, x^iy^j=0\}\quad \mbox{ and }\quad \cN^{(a)}_k=\{[x]\in \bP(R^a)\,|\, x^k=0\}.$$
We will refer to $\cN_k^{(a)}$ as {\bf nihilpotent loci} of $\bP(R^a)$. For brevity, we will set $\cN_k^{(1)}=\cN_k$. We will denote by $p_1$ and $p_2$ the standard projections from $\Gamma_{i,j}^{(a,b)}$ to $\bP(R^a)$ and $\bP(R^b)$, respectively. Notice that when we consider $\Gamma_{s,1}^{(a,b)}$, we have
$$p_1^{-1}([x])=\{([x],[y])\,|\, x^sy=0\}=[x]\times \bP(K^b_{x^s})$$
so all the fibers of $p_1$ are projective spaces (not necessarily of the same dimension).
\medskip

We will often assume the failure of some Lefschetz property and this assumption will reflect on a condition on the projection $p_1$. More precisely, we make the following observations.

\begin{itemize}
    \item Assume that $b\leq N/2$ and that $SLP_b(s)$ does not hold. Then, for all $[x]\in \bP(R^1)$ we have that the multiplication map $x^s\cdot :R^b\to R^{b+s}$ is not injective. In particular, there exists $[y]\in \bP(R^b)$ such that $x^sy=0$ in $R^{b+s}$. This shows that the failure of $SLP_b(s)$ is equivalent to ask that $p_1:\Gamma_{s,1}^{(1,b)}\to \bP(R^1)$ is surjective.
    \item Assume that $p_1:\Gamma_{s,1}^{(a,b)}\to \bP(R^a)$ is surjective. Then, as observed above, we have that all the fibers of $p_1$ are projective spaces and this implies that there exists a unique irreducible component $\Theta$ of $\Gamma_{s,1}^{(a,b)}$ which dominates $\bP(R^a)$ via $p_1$. In this case, we set 
    $$\pi_i=p_i|_{\Theta},\quad Y=p_2(\Theta)=\pi_2(\Theta)\qquad \mbox{ and } \qquad F_y=\pi_1(\pi_2^{-1}([y])) \mbox{ for all } [y]\in Y.$$
\end{itemize}

\begin{construction}
\label{CONSTR:OURFRAMEWORK}
To summarize, if $R$ is a SAGA of codimension $n+1$ and socle in degree $N$ and we assume that $SLP_b(s)$ does not hold for $R$, we can construct the loci $\Gamma_{s,1}^{(1,b)}, \Theta,Y$ and $F_y$ as above and we have the following diagram
\begin{equation}
\xymatrix@R=15pt@C=40pt{
F_y\times [y] \ar[ddd]_{\simeq} \ar@{^{(}->}[rd] \ar[rrr]^-{\pi_2} & & & [y] \ar@{^{(}->}[d] \\ 
& \Theta \ar@/_1.0pc/@{->>}[ddr]_-{\pi_1}  \ar@{->>}[rr]^-{\pi_2} \ar@{^{(}->}[rd]  & & Y \ar@{^{(}->}[d] \\
& & \Gamma_{s,1}^{(1,b)} \ar[r]^{p_2} \ar@{->>}[d]^-{p_1} & \bP(R^b) \\
F_y \ar@{^{(}->}[rr] & & \bP(R^1).
}
\end{equation}
We stress that, in this case, as $\Theta$ is the unique irreducible component which dominates $\bP(R^1)$, we have
$$\pi_1^{-1}([x])=p_1^{-1}([x])=[x]\times \bP(K^b_{x^s}) \quad \mbox{ for general }\quad [x]\in \bP(R^1).$$
On the contrary, for specific $[x]\in \bP(R^1)$, it can happen that $\pi_1^{-1}([x])\subsetneq p_1^{-1}([x])$ and that $\pi_1^{-1}([x])$ is not a projective subspace of $[x]\times \bP(R^b)$.
\end{construction}

In \cite{BFP} and in \cite{FP21} the authors have developped a technical lemma which gives strong informations on the locus $\Theta$ introduced above. We present now a generalised version of it.

\begin{proposition}
\label{PROP:KERCOKERDelux}
Let $T$ be an irreducible variety in $\bP(R^a)\times \bP(R^b)$ such that $p_1|_T:T\to \bP(R^a)$ is surjective. Assume that $T\subseteq \{x^iy^j=0\}=\Gamma_{i,j}^{(a,b)}$ with $i,j\geq 1$. Then
\begin{enumerate}
    \item For all $v\in R^a$ one has $vx^{i-1}y^{j+1}=0$;
    \item If $a(i)+b(j+1)\leq N$, then all points of $T$ satisfy also $x^{i-1}y^{j+1}=0$.
\end{enumerate}
In particular, if $a=b=1$ and $k\leq N-2$, if $T\subseteq \Gamma_{k,1}^{(1,1)}$ then $T\subseteq \Gamma_{i,j}^{(1,1)}$ for all $i,j$ such that $i+j=k+1$ and $j\geq 1$.
\end{proposition}

\begin{proof}
It is enough to prove the claim for a general smooth point $p=([x],[y])\in T$. For any $v\in R^a,t\in \bK$ consider $x'=x+tv\in R^a$. Since $p_1:T\to \bP(R^a)$ is surjective by hypothesis, we have that there exists $y'$ in $R^b\setminus\{0\}$ such that $(x')^i(y')^j=0$. Then we can define $\beta(t)$ such that $\beta(0)=y$ and $(x+tv)^i(\beta(t))^j=0$ for all $t\in \bK$. 
We can consider the expansion of $\beta$ and write this relation as
$$0\equiv(x+tv)^i(y+tw+t^2(\cdots))^j=x^iy^j+t(i vx^{i-1}y^j+j wx^iy^{j-1})+t^2(\cdots).$$

In particular we have $i vx^{i-1}y^j+j wx^iy^{j-1}=0$ for all $v\in R^a$. If we multiply by $y$ we have $i vx^{i-1}y^{j+1}=0$ for all $v\in R^a$ which yields the first claim since $i\geq 1$ by hyphotesis.
\vspace{2mm}

For the second claim, consider the multiplication map $R^a\times R^{(i-1)a+(j+1)b}\to R^{ia+jb+b}$ and notice that it is non degenerate by the assumption $ia+bj+b\leq N$. Hence, if $vx^{i-1}y^{j+1}=0$ for all $v\in R^a$, then one has also $x^{i-1}y^{j+1}=0$ as claimed.
\end{proof}

Proposition \ref{PROP:KERCOKERDelux} gives some easy but relevant relations involving the dimension of $Y,\Theta$ and $F_y$ for $[y]\in Y$ general. We summarize these results in the following

\begin{lemma}
\label{LEM:easyresults}
Assume that $R$ is any SAGA of codimension $n+1$ with socle in degree $N$ and assume that $SLP_b(k)$ does not hold for some integer $k$ such that $1\leq k\leq N-2b$. Then, defining $\Gamma_{k,1}^{(1,b)},\Theta,Y$ and $F_y$ as in \ref{CONSTR:OURFRAMEWORK}, we have the following properties:
\begin{enumerate}[(a)]
\item $\dim(Y)+\dim(F_y)=\dim(\Theta)\ge n=\dim(\bP(R^1))$;
\item if $b(k+1)\le N-1$ and $b\ge2$ then $0\le\dim(F_y)\le n-1$ and $1\le\dim(Y)\le\dim(\bP(R^b))-1$;
\item if $b=1$ and $k=N-2$ then $1\leq\dim(Y)\leq n-2\quad$ and $\quad 2\leq\dim(F_y)\leq n-1$.
\end{enumerate}
\end{lemma}

\begin{proof}
One can prove $(a)$ and $(b)$ by following the proof of \cite[Proposition 2.6]{BFP} and using Proposition \ref{PROP:KERCOKERDelux}. The last point has been proved in \cite[Corollary 2.11]{BFP}.
\end{proof}

Through the whole article, if $p\in X$ is a smooth point we will denote by $T_{X,p}$ the differential tangent space. The notation $T_p(X)$ will be used when $X$ is inside a projective space $\bP^n$ to mean its embedded Zariski tangent space. In particular, if $[x]\in X\subseteq\bP^n$ is a smooth point, and $\tilde{X}$ is the affine cone of $X$, we have $\bP(T_{\tilde{X},x})=T_{[x]}(X)$.
Let us now introduce a simple description of the tangent space for the nihilpotent loci which generalize \cite[Corollary 4.2]{BFP}. The proof is analogous with minimal changes. 

\begin{lemma}
\label{LEM:tangnihil}
Let $R$ be any SAGA. Then for $[\eta]\in\mathcal{N}_k^{(a)}$ general we have $T_{[\eta]}(\mathcal{N}_k^{(a)})\subseteq \bP(K^a_{\eta^{k-1}})$. If, moreover, $\eta^{k-1}\neq 0$, we have an equality.
\end{lemma}

Let us now show a technical result that holds for any SAGA when we deny the $SLP_1(k)$ for some $1\le k\le N/2$. Then, as in \ref{CONSTR:OURFRAMEWORK} we can construct $\Theta\subseteq \Gamma_{k,1}^{(1,1)}$, $Y$ and $F_y$.

\begin{lemma}
\label{LEM:xtangy}
With notations as above, if $p=([x],[y])\in\Theta$ is a general point, we have:
\begin{enumerate}[(a)]
\item $y\in T_{F_y,[x]}$ and $x\not\in T_{Y,[y]}$;
\item $T_{\tilde{F}_y,x} \subseteq K^{1}_{x^{\alpha} y^{\beta}}$ whenever $\alpha+\beta=k$ and $\beta\geq 1$, where $\tilde{F}_y$ denotes the affine cone over $F_y$.
\end{enumerate}
\end{lemma}

\begin{proof}
For $(a)$, let $p=([x],[y])\in \Theta$ be a general point. By \cite[Proposition 2.6]{BFP} we have that $F_y$ is a cone and $[y]$ is a vertex for it, so that the line $\langle [x],[y]\rangle$ is contained in $F_y$. This means that $[y]$ is a tangent vector in $[x]$, i.e. $[y]\in T_{F_y,[x]}$.
\medskip

For the second claim of $(a)$, let us suppose by contradiction that $[x]\in T_{Y,[y]}$. Let $\tilde{\Theta}$ be the lifting to $R^1\times R^1$ of $\Theta\subset \bP(R^1)\times \bP(R^1)$ and let $\tilde{\pi}_2$ be the projection on the second factor from $\tilde{\Theta}$. By construction, we have that
$\tilde{\Theta}\subseteq\{(x,y)\,|\, x^k y=0\}$.
Let $\tilde{Y}$ be the affine cone of $Y$. Since $\pi_2:\Theta\to Y$ is surjective, we have $\tilde{\pi}_2(\tilde{\Theta})=\tilde{Y}$ and that, for $p=(x,y)\in \tilde{\Theta}$ general,
$$d_p\tilde{\pi}_2:T_{\tilde{\Theta},p}\to T_{\tilde{Y},y}$$ is surjective.
\medskip

By assumption we have also that $x\in T_{\tilde{Y},y}$ so there exists a tangent vector to $\tilde{\Theta}$ at $p$ of the form $(v,x)$. Since points of $\tilde{\Theta}$ satisfy $x^ky=0$ we have 
$$0\equiv (x+tv+t^2(\cdots))^k(y+tx+t^2(\cdots))\quad (\mbox{mod }\, t^2)$$
which yields $x^{k+1}=0$. This is impossible by the generality of $x$: $[x]\not\in T_{Y,[y]}$.
\medskip

For $(b)$, first of all, notice that by Proposition \ref{PROP:KERCOKERDelux} we have $\tilde{F}_y \subseteq \{x\in R^1\,|\, x^{i}y^{j}=0\}$ for all $i+j=k+1$ and $i,j\geq 1$. Hence, for $p$ general, if an element $v\in R^1$ belongs to $T_{\tilde{F}_y,x}$ then the following relation must be satisfied
$$0\equiv (x+tv+t^2(\cdots))^iy^j=itvx^{i-1}y^j+t^2(\cdots) \quad (\mbox{mod }\, t^2).$$
Hence, we have that $v\in K^1_{x^{\alpha}y^{\beta}}$, with $\alpha+\beta=k$ and $\beta\ge 1$.
\end{proof}

Let us now consider a complete intersection SAGA $R$: we recall another technical result, dealing with the kernels of the multiplication maps presented above,  that will be used extensively in the whole paper. 

\begin{proposition} [Proposition 4.1 of \cite{BFP}]
\label{PROP:SameKer}
Let $R$ be any complete intersection SAGA of codimension $n+1$ presented in degree $d-1$.
Assume that $1 \leq s\leq N-1$. The following properties hold:
\begin{enumerate}[(a)]
\item If $[\eta]\in \bP(R^s)$, then $s\geq(d-2)\dim(K^1_{\eta})$;
\item Let $[\eta],[\zeta]\in \bP(R^s)$ and assume $s=(d-2)\dim(K^1_{\eta})=(d-2)\dim(K^1_{\zeta})$. Then
$K^1_\eta=K^1_\zeta$ if and only if $[\eta]=[\zeta]$.
\end{enumerate}
\end{proposition}

Finally, let us now focus on the specific case that will be treated in the next sections, that is the case where $R$ is a complete intersection SAGA of codimension $n+1$ presented in degree $d-1=2$. 

\begin{remark}
\label{REM:dimnihil}
If $d=3$, by Proposition \ref{PROP:SameKer} we get that for $[\eta]\in\bP(R^s)$ we have $\dim(K^1_{\eta})\le s$. From this, together with Lemma \ref{LEM:tangnihil}, we also get that $\dim(\mathcal{N}_k)\le k-2$.
\end{remark}


\section{Nihilpotent loci and geometric properties}
\label{SEC:Nihil}

In this section we will study geometric properties of the nihilpotent loci $\cN_k\subseteq \bP(R^1)$ where $R$ is a complete intersection SAGA of codimension $n+1$ presented by quadrics. Any $R$ as above has socle in degree $N=n+1$. We stress that we don't make any assumptions about the validity of any weak or strong Lefschetz property for $R$ in this section.
\medskip

We recall that the nihilpotent loci (in $\bP(R^1)$) are defined as
$$\cN_{k}=\{[x]\in \bP(R^1)\ |\ x^k=0\}.$$ 
If $X\subset \bP^r$ is non-empty, we denote by $\Sec^k(X)\subseteq \bP^r$ the $k$-secant variety associated to $X$, i.e.
$$\Sec^k(X):=\overline{\bigcup_{p_1,\dots,p_k\in X}{\langle p_1,\dots,p_k\rangle}}$$
where $\langle p_1,\dots,p_k\rangle$ is the linear span of the points $p_1,\dots,p_k$. For brevity, we set $\Sec^2(X):=\Sec(X)$. The interested reader can refer to \cite[Chapter 1]{Rus16} for various properties of these classical loci (although the definition considered is slightly different from the one adopted by us).
\medskip

We stress that, like the non-Lesfschetz loci $\bP(R^1)\setminus \cL_k$, the nihilpotent loci $\cN_k$ are expected to be empty when $k$ is small for $R$ "general". Hence it is interesting to study these loci when $R$ is "special". For example, these loci give a lot of information for SAGAs for which some Lefschetz properties do not hold.
\medskip

Let us start by analyzing the locus $\cN_2\subseteq \bP(R^1)\simeq \bP^n$. We recall that by Lemma \ref{LEM:tangnihil} we have that $\dim(\cN_2)\leq 0$ so it is either empty or it is the union of a finite number of points. These points have to satisfy the following:

\begin{proposition}
\label{PROP:Y2points}
Assume that $[t_1],\dots,[t_k]\in \cN_2$ are distinct points. Then $\Pi_{i=1}^kt_i\neq 0$ in $R$ and $[t_1],\dots, [t_k]$ are in general position in $\bP(R^1)$. In particular, $\#\cN_2\leq n+1=N$.
\end{proposition}

\begin{proof}
The statement is trivially true for $k=1$. If $k=2$ the only statement one has to check is that $t_1t_2\neq 0$. This is true since $K^1_{t_1}=\langle t_1\rangle$ by Proposition \ref{PROP:SameKer} and $[t_1]\neq [t_2]$.
We will then proceed by induction assuming that the claim is true till $k-1$. 

Let $T=\{[t_1],\dots,[t_k]\}$ be a set of $k$ distinct points of $\cN_2$. By contradiction, let us assume that either $(A_1)$ or $(A_2)$ holds, where
\begin{description}
\item [$(A_1)$] $\{t_1,\dots,t_{k}\}$ are linearly dependent
\item [$(A_2)$] $\Pi_{i=1}^kt_i=0$.
\end{description}
First of all, we claim that $(A_2)$ is equivalent to $(A_1)$. By induction hypothesis, for $\{z_1,\dots,z_{k-1}\}\subset \{t_1,\dots, t_{k}\}$ with $[z_i]\neq [z_j]$ for all $i\neq j$, we have $\Pi_{i=1}^{k-1}z_i\neq 0$, so $K^1_{z_1\cdots z_{k-1}}$ has dimension at most $k-1$ by Proposition \ref{PROP:SameKer}. 
Since $z_i^2=0$ by assumption, we have $K^1_{z_1\cdots z_{k-1}}=\langle z_1,\dots, z_{k-1}\rangle$. Then, $(A_1)$ holds if and only if we have, up to a permutation of the elements, $t_k\in \langle t_1,\dots, t_{k-1}\rangle=K^1_{t_1\cdots t_{k-1}}$ and this is equivalent to $\Pi_{i=1}^kt_i=0$, i.e. $(A_2)$.
\medskip

Hence, let us suppose that $t_k\in\langle t_1,\dots,t_{k-1}\rangle$, so we can write $t_k=\sum_{i=1}^{k-1}a_it_i$. Then we have $$0=t_k^2=2\sum_{1\le i<j\le k-1}a_ia_jt_it_j.$$
If $k=3$ we have $0=t_3^2=2a_1a_2t_1t_2$ so, since $t_1t_2\neq 0$ by induction hypothesis, we have either $a_1=0$ or $a_2=0$. This implies either $\{t_1,t_3\}$ or $\{t_2,t_3\}$ linearly dependent, and we get a contradiction since this is against the induction hypothesis.
If $k\geq 4$, by multiplying by $\Pi_{i=1}^{k-3}t_i$, we get
$$0=2a_{k-2}a_{k-1}\Pi_{i=1}^{k-1}t_i.$$
Since $\Pi_{i=1}^{k-1}t_i\neq 0$ by induction hypothesis, we have either $a_{k-2}=0$ or $a_{k-1}=0$ and we have a contradiction as in the case $k=3$.
\end{proof}

By considering the Fermat hypersurface $X=V(F)$ in $\bP^n$, one can easily see that, for the Jacobian ring $R=S/J(F)$, the set $\cN_2$ consists of exactly $n+1$ independent points. However, also the converse is true, as shown by the following:

\begin{corollary}
\label{COR:ReconstructFermat}
Assume that $\#\cN_2=n+1$. Then $R$ is the Jacobian ring of a cubic hypersurface $X$ projectively equivalent to the Fermat cubic hypersurface in $\bP^n$.
\end{corollary}

\begin{proof}
By assumption we have that $\cN_2=\{[t_0],\dots, [t_n]\}$. By Proposition \ref{PROP:Y2points}, $\{t_0,\dots,t_n\}$ are $n+1$ linearly independent forms so 
$R=S/I$ with $S=\bK[t_0,\dots, t_n]$. On the other hand, in $S$ we have $t_i^2\in I$ and $\{t_0^2,\dots,t_n^2\}$ is a regular sequence which generates $I$ as ideal of $S$. Then, if we set $F=\sum_{i=0}^n t_i^3$, we have that $I$ is the Jacobian ideal of the Fermat cubic hypersurface $X=V(F)$ as claimed.
\end{proof}

\begin{remark}
We have $\Sec^k(\cN_2)\subseteq \cN_{k+1}$. Indeed, if $[t_1],\dots, [t_{k}]\in \cN_2$ we have $t_i^2=0$. In particular, every monomial of degree $k+1$ in the variables $t_i$ is identically $0$. Then $(\sum_{i=1}^{k}a_it_i)^{k+1}\equiv 0$ for all $a_1,\dots, a_k\in \bK$ so $\Sec^k(\cN_2)\subseteq \cN_{k+1}$. More generally, 
$$\mbox{ whenever } r>k(a-1) \mbox{ one has } \Sec^k(\cN_a)\subseteq \cN_{r}.$$
Indeed, consider $[t_1],\dots,[t_k]\in \cN_a$ and let $m=\prod_{i=1}^kt_i^{\alpha_i}$ with $\sum_{i=1}^{k}\alpha_i=r$. We have $m=0$ if there exists $i$ such that $\alpha_i\geq a$. On the other hand, this always happens if $r>k(a-1)$: if $\alpha_i\leq(a-1)$ for all $i$, we would have 
$$r=\sum_{i=1}^k\alpha_i\leq \sum_{i=1}^k(a-1)=k(a-1)<r$$
which gives a contradiction.
\end{remark}

\begin{lemma}
\label{LEM:LINESinZ3}
If $L$ is a line contained in $\cN_3$, then $L\subseteq \Sec(\cN_2)$, i.e. a line in $\cN_3$ is a line joining two different points of $\cN_2$.
\end{lemma}

\begin{proof}
Assume that $L$ is a line in $\cN_3$. Since the dimension of $\cN_3$ is at most $1$ by Remark \ref{REM:dimnihil}, we have that $L$ is a component of $\cN_3$. As $\dim(\cN_2)\leq0$ and $\dim(K^1_{x})\leq 1$ for any $x\in R^1$ by Proposition \ref{PROP:SameKer}(a), we can find $[v],[w]\in L$ such that $[v]\neq [w]$, $[v],[w]\not\in \cN_2$ and $vw\neq 0$. By hypothesis, we have that $(v+tw)^3=0$ for all $t\in\bK$ so
$v^3=v^2w=vw^2=w^3=0$. Then, $K^1_{v^2},K^1_{w^2}$ and $K^1_{vw}$ contain $\langle v,w\rangle$. On the other hand, these subspaces have dimension at most $2$ by Proposition \ref{PROP:SameKer}(a) so they coincide with $\langle v,w\rangle$.
By Proposition \ref{PROP:SameKer}(b), there exist $\lambda,\mu\in\bK$ such that 
\begin{equation}
\label{EQ:lambdamu}
v^2=\lambda w^2 \qquad \mbox{ and }\qquad vw=\mu w^2
\end{equation}
so we have $(v+tw)^2=v^2+2tvw+t^2w^2=w^2(t^2+2\mu t+\lambda)$.
\medskip

We claim that $t^2+2\mu t+\lambda$ has two distinct roots so $L$ is indeed a line contained in $\Sec(\cN_2)$. Assume, on the contrary, that $t^2+2\mu t+\lambda$ is a square. This implies that $\mu^2=\lambda$. Then, from the Equations \eqref{EQ:lambdamu}, we obtain
$$v(v-\mu w)=0 \qquad w(v-\mu w)=0$$
so $v-\mu w\in K^1_{v}\cap K^1_{w}$. By Proposition \ref{PROP:SameKer} we can conclude that $[v]=[w]$ which is against our assumptions.
\end{proof}

We will generalize this result in Theorem \ref{THM:SEC} by considering suitable linear subspaces contained in $\cN_k$. We need first the following technical lemma.

\begin{lemma}
\label{LEM:indpoint}
Let $k\geq 2$ and let $T$ be an hypersurface in $\bP^{k}\subset \bP(R^1)$. Assume either that 
\begin{enumerate}
    \item $0\leq s\leq k-1$ or
    \item $s=k$ and the support of $T$ is not contained in the union of $2$ different hyperplanes.
\end{enumerate}
Then there exist $[x_0],\dots,[x_{s}]\in T$ which are linearly independent and such that $\Pi_{i=0}^{s}x_i\neq 0$.
\end{lemma}

\begin{proof}
The statement of the lemma is clearly true for $s=0$. We will proceed by induction on $s\leq k$. Then assume that there are $[x_0],\dots,[x_{s-1}]\in T$ which are linearly independent and with $y=x_0\cdots x_{s-1}\neq 0$. Consider the linear spaces $\tau_1=\langle x_0,\dots,x_{s-1}\rangle$ and $\tau_2=\bP(K^1_{y})$. We are done if we prove that $U=T\setminus (\tau_1\cup \tau_2)$ is not empty.
By construction we have $\dim(\tau_1)=s-1$ and $\dim(\tau_2)\leq s-1$ by Proposition \ref{PROP:SameKer}. Hence, if $s<k$, $U$ is an open dense subset of $T$. If $s=k$ and the support of $T$ is not contained in the union of $2$ different hyperplanes, there exists an irreducible component $C$ of $T$ which is different from $\tau_1$ and $\tau_2$. Then $C\setminus (\tau_1\cup \tau_2)$ is not empty so $U$ is again not empty as claimed.
\end{proof}

\begin{theorem}
\label{THM:SEC}
Assume that $\pi$ is a $(k-1)$-plane contained in $\cN_{k+1}$. Then 
\begin{enumerate}[(A)]
    \item $T_k=\pi\cap \cN_k$ is an hypersurface of degree $k$ (with possible multiple components) in $\pi$;
    \item there exist $[x_0],\dots,[x_{k-1}]$ in $T_k$ which are linearly independent, $\Pi_{i=0}^{k-1}x_k\neq 0$.
\end{enumerate}
In particular, $T_k$ is non degenerate in $\pi$ and $\pi\subseteq \Sec^k(\cN_k)$. 
\end{theorem}

\begin{proof}
Notice, first of all, that $\dim(\cN_{k})\leq k-2$ by Remark \ref{REM:dimnihil}, so $\pi\setminus \cN_k\neq \emptyset$. Then, we can find  $\{x_0,\dots,x_{k-1}\}$ linearly independent which span $\pi$ and such that $x_{k-1}\not\in \cN_k$. Since $\pi\subseteq \cN_{k+1}$, we have that
$$(\alpha_0x_0+\cdots +\alpha_{k-1}x_{k-1})^{k+1}\equiv 0\qquad \forall \alpha_0,\dots, \alpha_{k-1}\in \bK$$
and this is equivalent to say that all monomials of degree $k+1$ in the variables $x_0,\dots,x_{k-1}$ are $0$. Then, if $m$ is a monomial of degree $k$ in these variables, either $m=0$ or $K^1_{m}=\langle x_0,\dots,x_{k-1}\rangle$, by Proposition \ref{PROP:SameKer}.
In particular, we have that for each monomial of degree $k$ there exists $\lambda_m\in\bK$ with $m=\lambda_m x_{k-1}^k$ (recall that we assumed $x_{k-1}\not\in \cN_{k}$). 
Then
\begin{equation}
\label{EQ:Degk}
(\alpha_0x_0+\cdots +\alpha_{k-1}x_{k-1})^{k}=p_{k}(\underline{\alpha}) x_{k-1}^k
\end{equation}
where $p_k(\underline{\alpha})$ is a homogeneous polynomial of degree $k$ in the variables $\alpha_0,\dots, \alpha_{k-1}$. It is not $0$ since the coefficient of $\alpha_{k-1}^k$ is $1$ by construction. By Equation \eqref{EQ:Degk}, $T_k=\pi\cap \cN_{k}$ is described by the vanishing  of $p_k(\underline{\alpha})$. In particular, $\cN_k$ is not empty and we have also proved $(A)$.
\medskip

For $(B)$, if $T_k$ has support which is not contained in $2$ different hyperplanes, the thesis follows directly from Lemma \ref{LEM:indpoint} so we have to discuss only the cases 
$$(B_1): \ \Supp(T_k)=H_1\cup H_2\qquad \mbox{ and }\qquad (B_2): \ \Supp(T_k)=H_1$$
where $H_1$ and $H_2$ are distinct hyperplanes.
\medskip

In both cases $(B_1)$ and $(B_2)$, there is an hyperplane $H_1$ of $\pi$ contained in $T_k$. We recall that $T_k$ is contained in $\cN_k$ by construction. By Lemma \ref{LEM:indpoint} applied to $H_1\subset \pi$ we can find $[x_0],\dots,[x_{k-2}]\in H_1$ which are linearly independent and such that $y=\prod_{i=0}^{k-2}x_i\neq 0$. Since $H_1=\langle [x_0],\dots,[x_{k-2}]\rangle$ and $H_1\subset \cN_k$ we have that all monomials of degree $k$ in the variables $x_0,\dots,x_{k-2}$ are $0$. Then, by Proposition \ref{PROP:SameKer}, $K^1_{y}=\langle x_0,\dots,x_{k-2}\rangle$ so $H_1=\bP(K^1_{y})$. 
\medskip

If we are in case $(B_1)$ we can then choose $x_{k-1}$ in $H_2\setminus H_1$ and $\{x_0,\dots,x_{k-2},x_{k-1}\}$ is a set of points with the desired properties. We claim now that case $(B_2)$ can not occur. Assume, by contradiction, that $\Supp(T_k)$ is the hyperplane $H_1=\bP(K^1_{y})$. Then for any $x_{k-1}$ in $\pi\setminus H_1$ we have that $\pi=\langle x_0,\dots,x_{k-2},x_{k-1}\rangle$, $x_{k-1}^k\neq 0$ and $yx_{k-1}\neq 0$. With this choice of the $x_i$'s, the polynomial $p_{k}(\underline{\alpha})$ of Equation \eqref{EQ:Degk} is proportional to $\alpha_{k-1}^k$ since $T_k=\pi\cap \cN_k$ has support on $H_1$. On the other hand the coefficient of $\prod_{i=0}^{k-1}\alpha_i$ can not be zero since $\prod_{i=0}^{k-1}x_i=yx_{k-1}\neq 0$.
\end{proof}

We conclude this section by presenting some examples in order to make the phenomenology of the nihilpotent loci clearer (some computations have been made by using the computer algebra software Magma). We set $S=\bK[x_0,\dots,x_n]=\bigoplus_{k\geq 0} S^k$ and we define $\{w_0,\dots,w_n\}$ to be the projective coordinates on $\bP(R^1)$ induced by the basis $\{x_0,\dots,x_n\}$ of $R^1=S^1$.

\begin{example}
\label{EX:EX1}
Let $X$ be the Fermat cubic in $\bP^n$ and consider the Jacobian ring $R$ of $X$. For any $2\leq k\leq n$ we have 
$\left(\sum_{i}w_ix_i\right)^k\in J^k$ if and only if all monomials in the $\{w_i\}$ of degree $k$ without multiple factors vanish. This is true whenever any set of $n-k+1$ variables is zero. With these arguments one can prove that $\cN_k$ is the union of the coordinated planes of dimension $k-2$. In particular, $\cN_k=\Sec^k(\cN_2)$ and $\Sing(\cN_k)=\cN_{k-1}$ for $k\geq 2$.
\end{example}

\begin{example}
\label{EX:EX2}
Consider the smooth cubic surface $X=V(f)$ with $f=x_0^3+x_1^3+x_2^3+x_3^3+6x_0x_1x_2$ and consider the Jacobian ring $R$ of $X$.
One has that there are $4$ points in $\cN_2$ so, by Corollary \ref{COR:ReconstructFermat}, $X$ is the Fermat cubic up to a projective transformation. Indeed, if $\lambda$ is a non trivial third root of $1$, we have
$$(x_0+x_1+x_2)^3+(x_0 - (\lambda+1)x_1 + \lambda x_2)^3+(x_0 + \lambda x_1 - (\lambda+1)x_2)^3+3x_3^3=3f.$$
\end{example}

\begin{example}
\label{EX:EX3}
Consider the smooth cubic surface $X=V(f)$ with $f=x_0^3+x_1^3+x_2^3+x_3^3+3x_0x_1x_2$ and consider the Jacobian ring $R$ of $X$.
If $P=[0:0:0:1]$ and $C=V(w_3,g)$ is the smooth plane cubic with $g:=w_0^3+w_1^3+w_2^3-6w_0w_1w_2$, we have (considering the reduced structure)
$$\cN_2=\{P\}\qquad \cN_3=\{P\}\cup C\qquad \cN_4=V(w_3)\cup V(g).$$
In particular, $\cN_4$ is the union of the plane containing the cubic curve $C$ and the cone with vertex $P$ generated by $C$. Notice that $\cN_3$ does not have pure dimension.
\end{example}

\begin{example}
\label{EX:EX4}
Consider the smooth cubic surface $X=V(f)$ with $f=x_0^3+x_1^3+x_2^3+x_3^3+x_0(x_1^2+x_2^2+x_3^2)$ and let $R$ be its Jacobian ring. One can show that, in this case, $\cN_2$ and $\cN_3$ are both empty whereas $\cN_4$ is a smooth quartic hypersurface.
\end{example}

\begin{example}
\label{EX:EX5}
Consider the regular sequence $\{x_0^2,x_1^2,x_2^2,x_3^2+2x_0x_1\}$ in $S=\bK[x_0,\dots,x_3]$, the ideal $J$ spanned by it and set $R=S/J$. Notice that $J$ is not the Jacobian ideal of a cubic surface.
Let $P_i$ be the coordinated points and consider the conic $C=V(w_2,g)$ with $g:=w_3^2-3w_0w_1$. Then we have
$$\cN_2=\{P_0,P_1,P_2\}\qquad \cN_3=\langle<P_0,P_1>\rangle\cup \langle<P_0,P_2>\rangle\cup\langle<P_1,P_2>\rangle\cup C$$
$$\cN_4=V(w_3)\cup V(w_2)\cup V(g).$$
In particular, $\cN_4$ is the union of two planes (the first one - $V(w_3)$ - contains $P_0,P_1$ and $P_2$ and the lines joining these points whereas the second - $V(w_2)$ - is the plane containing the conic $C$ and the line $\langle P_0,P_1\rangle$) and $V(g)$ (which is a quadric cone with vertex $P_2$). Notice that, as varieties, we have $\Sing(\cN_k)=\cN_{k-1}$ for $k=2,3,4$.
\end{example}


\section{Strong Lefschetz property in degree $1$ for codimension $6$}
\label{SEC:SLP1}

In this section we prove our main result, i.e. the following:
\begin{theorem}
\label{THM:SLP1}
Let $R$ be a complete intersection SAGA presented by quadrics of codimension $6$. Then $R$ satisfies the strong Lefschetz property in degree 1 ($SLP_1$), i.e. the general element $x\in R^1$ is such that the map $x^4\cdot:R^1\rightarrow R^5$ is an isomorphism.
\end{theorem}

In order to prove the above theorem we will need some technical results (Propositions \ref{PROP:Fyhyp}, \ref{PROP:dimYdiversoda_nmeno2} and \ref{PROP:2,3}) which will be treated separately in Subsection \ref{SUBSEC:prepSLP1} as they are interesting on their own because they are valid in a much larger framework than the one of the theorem. Now we will assume them and we will deal with the proof of the main theorem.

\begin{proof} (of Theorem \ref{THM:SLP1})
Let us assume by contradiction that the statement does not hold: the map $x^4\cdot:R^1\rightarrow R^5$ is never injective for $x\in R^1$. We are then in the situation described in Construction \ref{CONSTR:OURFRAMEWORK}
with $\Gamma=\Gamma_{4,1}^{(1,1)}=\{([x],[y])\in\bP(R^1)\times\bP(R^1) \ | \ x^4y=0\}$: we define $\Theta,Y$ and $F_y$ as usual.
Let us now focus on the dimensions of $Y$ and of $F_y$ for general $[y]\in Y$.

First of all, let us recall that, by Lemma \ref{LEM:easyresults}, we have 
\begin{equation}
\label{EQ:bounds}
1\le\dim(Y)\le3 \qquad 2\le\dim(F_y)\le4 \qquad \dim(Y)+\dim(F_y)=\dim(\Theta)\ge5.
\end{equation}
Let us now list in the following table the possible values of the pairs $(\dim(Y),\dim(F_y))$.
By using the constraints \eqref{EQ:bounds} and the various results proved in Subsection \ref{SUBSEC:prepSLP1}, we can exclude all the cases: in the table below we specify which result rules out each pair.

\begin{center}
\begin{tabular}{c||c|c|c|}
$\dim(Y)$ VS $\dim(F_y)$ & $2$ & $3$ & $4$\\
\hline\hline
1 & \eqref{EQ:bounds} & \eqref{EQ:bounds} & Prop. \ref{PROP:Fyhyp}\\
\hline
2 & \eqref{EQ:bounds} & Prop. \ref{PROP:2,3} + Rem. \ref{REM:dimnihil} & Prop. \ref{PROP:Fyhyp}\\
\hline
3 & Prop. \ref{PROP:dimYdiversoda_nmeno2} & Prop. \ref{PROP:dimYdiversoda_nmeno2} & Prop. \ref{PROP:Fyhyp} \\ \hline
\end{tabular}
\end{center}

Since no pair as above is possible for our framework, we get a contradiction and this concludes the proof.
\end{proof}

We have the following easy but important consequence.

\begin{corollary}
The jacobian ring of a smooth cubic fourfold satisfies the strong Lefschetz property in degree $1$.
\end{corollary}


\subsection{General preparatory results}
\label{SUBSEC:prepSLP1}

In this subsection we consider a complete intersection SAGA presented in degree $d-1$, i.e. $R=S/I$ where $S=\bK[x_0,\dots,x_n]$ and $I$ is generated by a regular sequence of forms of degree $d-1$.
In this case, we recall that $R$ is a SAGA of codimension $n+1$ with socle in degree $N=(d-2)(n+1)$. 
Moreover, we will assume that $R$ is a SAGA which does not satisfy the strong Lefschetz property in degree $1$ at range $k$ with $2\leq k\leq N-2$ ($SLP_1(k)$). Equivalently, the multiplication map 
$x^k\cdot :R^1\rightarrow R^{k+1}$ is never injective. Hence, we are in the situation described more generally in Section \ref{SEC:Notations}. Under the above assumptions we have 
\begin{equation}
\label{DIAG:k}
\xymatrix@R=10pt{
F_y\times [y] \ar[ddd]_{\simeq} \ar@{^{(}->}[rd] \ar[rrr]^-{\pi_2} & & & [y] \ar@{^{(}->}[d] \\ 
& \Theta \ar@/_1.0pc/@{->>}[ddr]_-{\pi_1}  \ar@{->>}[rr]^-{\pi_2} \ar@{^{(}->}[rd]  & & Y \ar@{^{(}->}[d] \\
& & \Gamma_{k} \ar[r]^-{p_2} \ar@{->>}[d]^-{p_1} & \bP(R^1) \\
F_y \ar@{^{(}->}[rr] & & \bP(R^1)
}
\end{equation}
where we have set $\Gamma_{k}:=\Gamma_{k,1}^{(1,1)}=\{([x],[y])\in\bP(R^1)\times\bP(R^1) \ | \ x^ky=0\}$. We recall that $\Theta$ is the unique irreducible component of $\Gamma_{k}$ that dominates $\bP(R^1)$ via its first projection $\pi_1$, $Y=\pi_2(\Theta)$ and $F_y=\pi_1(\pi_2^{-1}([y]))$ for $[y]\in Y$.

\begin{remark}
\label{REM:WLP1}
We recall that all complete intersection SAGAs presented in degree $d-1$ satisfy $SLP_1(1)=WLP_1$ (see \cite[Proposition 4.3]{MN13} and \cite[Corollary 4.3]{BFP}). For this reason in this subsection we do not consider $k=1$ since in this case it is possible to construct $\Gamma_{1}$, but $p_1$ is never dominant (so $\Theta,Y$ and $F_y$ cannot be constructed).
\end{remark}

Now we prove some results giving restrictions on the dimensions of $Y$ and of the general $F_y$ with $[y]\in Y$. We are ultimately interested into the case where $d=3$. Nevertheless, we stress that the following proposition (\ref{PROP:Fyhyp}) holds for every $d\geq 3$.


\begin{proposition}
\label{PROP:Fyhyp}
If we assume $n>\frac{k}{d-2}$ and $k\geq 2$, then
$$\dim(F_y)\le n-2.$$
In particular, if $d=3$, then $F_y$ cannot be an hypersurface ($k\le N-2$).
\end{proposition}

\begin{proof}
Recall that $\dim(F_y)\le n-1$ by Lemma \ref{LEM:easyresults} so we have to rule out only the case $\dim(F_y)=n-1$.

Let us assume, by contradiction, that $F_y$ is an hypersurface. Hence, by denoting with $\tilde{F}_y$ the affine cone over $F_y$, we have $\dim(\tilde{F}_y)=n$.
\medskip

Recall that $\Theta:=\Theta_k\subseteq \Gamma_k=\{([x],[y])\,|\, x^ky=0\}$ by assumption. We will show that the multiplication map $x^{k-1}\cdot:R^1\to R^{k}$ is never injective so we can define, as we have done for $\Theta_k$, an incidence correspondence $\Gamma_{k-1}$ with a unique irreducible component $\Theta_{k-1}$ which dominates $\bP(R^1)$ via its first projection. Moreover, we will have $\Theta_k=\Theta_{k-1}$ so $F_y$ is also the fiber of the second projection from $\Theta_{k-1}$ and we can iterate this process.
\medskip

We claim now that $\Theta\subseteq \Gamma_{k-1}$. If $p=([x],[y])\in\Theta$ is general, by using Lemma \ref{LEM:xtangy} and Proposition \ref{PROP:SameKer} we can conclude 
$$n=\dim(T_{\tilde{F_y},x})\le\dim(K^1_{x^{k-1}y})\leq \frac{k}{d-2}$$
unless $x^{k-1}y=0$. Since, by hypothesis, we have that $n>k/(d-2)$, the only possibility is that $x^{k-1}y=0$. In particular we have shown that $\Theta\subseteq \{([x],[y])\,|\, x^{k-1}y=0\}=\Gamma_{k-1}$ as claimed.
\medskip

Then, $\Theta$ is contained in $\Theta_{k-1}$ since it dominates $\bP(R^1)$. On the other hand, since $\Gamma_{k-1}\subseteq \Gamma_{k}$, we have also the other inclusion: $\Theta=\Theta_{k-1}$. In particular, the varieties $Y$ and $F_y$ defined for $\Theta$ are the same as the ones defined for $\Theta_{k-1}$. Then, by reasoning as before, we obtain $n=\dim(T_{\tilde{F}_y,x})\leq \dim(K^1_{x^{k-2}y})$. If we assume that $x^{k-2}y\neq 0$ for $p$ general in $\Theta$, by Proposition \ref{PROP:SameKer} we would obtain $n\leq (k-1)/(d-2)\leq k/(d-2)$ which is, as before, incompatible with the hypothesis on $n$. Then $\Theta\subseteq \Gamma_{k-2}$ and we can iterate this process.
\medskip

By recursion, we reduce ourselves to the case with $k=1$. We can then see $\Theta$ as a subvariety of $\Gamma_1$ which dominates $\bP(R^1)$ via its first projection. This implies the failure of the weak Lefschetz Property in degree $1$. Then, by Remark \ref{REM:WLP1}, we get a contradiction: $F_y$ has dimension at most $n-2$, as claimed.
\end{proof}


\begin{proposition}
\label{PROP:dimYdiversoda_nmeno2}
Assume that $d=3$. Then, the dimension of $Y$ is at most $n-3$.
\end{proposition}

\begin{proof}
First of all, let us notice that if $k\le N-3=n-2$, then by Proposition \ref{PROP:KERCOKERDelux}, $Y$ is contained in $\cN_{k+1}\subseteq \cN_{n-1}$, whose dimension is at most $n-3$ (see Lemma \ref{LEM:tangnihil}). Hence, we easily get that in this case $\dim(Y)\le n-3$ as claimed.
\medskip

Let us now consider the remaining case: $k=N-2=n-1$. Recall that $\dim(Y)\leq n-2$ by Lemma \ref{LEM:easyresults} so, to prove the proposition, we only have to rule out the case $\dim(Y)=n-2$. Assume by contradiction that $\dim(Y)=n-2$. Since $k=n-1$, proceeding as above, we have that $Y$ is contained in $\cN_n$, which has dimension at most $n-2$. Hence, we get that $Y$ is a component of $\cN_n$. Then, if $[y]\in Y$ is a general point ($y^{n-1}\ne0$, since $Y\not\subseteq \cN_{n-1}$, for dimension reasons), we can write $T_y(Y)=T_y(\cN_{n})=\bP(K^1_{y^{n-1}})$ by Lemma \ref{LEM:tangnihil}. However, the general $[x]\in F_y$ belongs to $\bP(K^1_{y^{n-1}})$ and then to $T_y(Y)$, contradicting Lemma \ref{LEM:xtangy}.
\end{proof}


\begin{proposition}
\label{PROP:2,3}
Assume that $d=3$. Then, the following conditions are not compatible: 
\begin{enumerate}[(a)]
\item for $y\in Y$ general, $\dim(F_y)=k-1$;
\item $Y\not\subseteq \cN_{k-1}$.
\end{enumerate}
\end{proposition}

\begin{proof}
Notice that for $k=1$, condition $(a)$ cannot hold since otherwise we would have $\dim(Y)=n$, which is impossible. Hence we can assume that $k\geq 2$.
\medskip

Let us observe that by Proposition \ref{PROP:KERCOKERDelux} we have that $x^{\alpha}y^{\beta}=0$ for every $\alpha,\beta$ such that $\alpha+\beta=k+1$ and $\beta\ge1$; in particular we have $y^{k+1}=0$ and $Y\subseteq \cN_{k+1}$.
\medskip

Let us assume by contradiction that both conditions $(a)$ and $(b)$ hold. By $(b)$ we have $y^{k-1}\ne0$ for $y\in Y$ general so for $([x],[y])$ general in $\Theta$ we have $xy^{k-1}\ne0$. Indeed, otherwise, $F_y$ would be contained in $\bP(K^1_{y^{k-1}})$, whose dimension is at most $k-2$, which is impossible by assumption. As a consequence, we have that
\begin{equation}
\label{EQ:2.3_1}
\mbox{for }\quad ([x],[y])\in \Theta\quad \mbox{general, }\quad x^{\alpha}y^{\beta}\ne0\quad \mbox{ for }\quad \alpha+\beta=k \quad \mbox{ with }\quad \alpha,\beta\ge1
\end{equation}
since, otherwise, by using Proposition \ref{PROP:KERCOKERDelux}, we would obtain also $xy^{k-1}=0$.
\medskip

By property \eqref{EQ:2.3_1} and since $\dim(F_y)=k-1$ by assumption, we also have
\begin{equation}
\label{EQ:2.3_2}
T_x(F_y)=\bP(K^1_{x^{k-1}y})=\bP(K^1_{x^{k-2}y^2})
\end{equation}
by Lemma \ref{LEM:xtangy}.

Let us now claim that 
\begin{equation}
\label{EQ:2.3_3}
\mbox{for }\quad ([x],[y])\in \Theta\quad \mbox{general, }\quad T_y(Y)\subseteq T_x(F_y).
\end{equation}

To show this, first of all, recall that $Y\not \subseteq \cN_{k-1}$ and $Y\subseteq \cN_{k+1}$. Let us now consider two cases:
$$1)\quad Y\not\subseteq \cN_k\qquad \qquad \mbox{ and }\qquad\qquad  2)\quad Y\subseteq \cN_k.$$
Assume that $p=([x],[y])\in \Theta$ is general (so that $[y]$ is general in $Y$ and $[x]$ is general in $F_y$).
In the first case, since $Y$ is contained in $\cN_{k+1}$, we have $T_y(Y)\subseteq\bP(K^1_{y^k})$ by Lemma \ref{LEM:tangnihil}. Moreover, we have that $\bP(K^1_{y^k})=T_x(F_y)$ since $y^k\ne0$ and $\dim(F_y)=k-1$. Analogously, for the second case we have
$T_y(Y)\subseteq\bP(K^1_{y^{k-1}})\subseteq \bP(K^1_{xy^{k-1}})=T_x(F_y)$.
Here, we have used that $xy^{k-1}\neq 0$ since $p$ is general (by property \eqref{EQ:2.3_1}).
\medskip

Consider, as in Lemma \ref{LEM:xtangy}, the affine cone $\tilde{Y}$ of $Y$, the lifting $\tilde{\Theta}$ of $\Theta\subset \bP(R^1)\times \bP(R^1)$ to $R^1\times R^1$ and its projection $\tilde{\pi}_2$ on the second factor. By construction, we have that
$\tilde{\Theta}\subseteq \{(x,y)\,|\, x^\alpha y^\beta=0\}=\tilde{\Gamma}_{\alpha,\beta}$ whenever $\alpha+\beta=k+1$ and $\beta\geq 1$.
As in Lemma \ref{LEM:xtangy}, for $p=(x,y)\in \tilde{\Theta}$ general, the differential map
$$d_p\tilde{\pi}_2:T_{\tilde{\Theta},p}\to T_{\tilde{Y},y}$$ is surjective.

Let us take any $w$ in $T_{\tilde{Y},y}$. By Property \eqref{EQ:2.3_3}, we have that its class $[w]$ belongs to $T_x(F_y)$. Moreover, by the surjectivity of $d_p\tilde{\pi}_2$, we can take in $T_{\tilde{\Theta},p}$ an element of the form $(v,w)$.

The tangent space to $\tilde{\Theta}$ in $p$ is a subspace of the tangent space $T_{\alpha,\beta}=T_{\tilde{\Gamma}_{\alpha,\beta},p}$ to the locus $\tilde{\Gamma}_{\alpha,\beta}$ in $p$, so we have $(v,w)\in T_{\alpha,\beta}$. Then, we have 
$$0\equiv (x+tv+t^2(\cdots))^\alpha(y+tw+t^2(\cdots))^\beta\, (\mbox{mod }\, t^2)$$
for $\alpha,\beta$ as above. In particular, by taking $(\alpha,\beta)$ equal to $(k-1,2)$ and $(k,1)$ we obtain the following relations satisfied by $(v,w)$:
$$(k-1)vx^{k-2}y^2+2x^{k-1}yw=0 \qquad kvx^{k-1}y+x^kw=0.$$
Since $[w]\in T_x(F_y)$, by property \eqref{EQ:2.3_2}, we have $x^{k-1}yw=0$. Then, from the first equation we get $v\in K^1_{x^{k-2}y^2}=K^1_{x^{k-1}y}$ (again by property \eqref{EQ:2.3_2}).
Hence the second equation yields $x^kw=0$. In conclusion, we have proven that $T_{y}(Y)\subset \bP(K^1_{x^k})=\pi_1^{-1}([x])$. We stress that the last equality holds since $[x]$ is general and then, the whole fiber over $[x]$ with respect to $p_1$ is contained in $\Theta$ so $\bP(K^1_{x^k})=p_1^{-1}([x])=\pi_1^{-1}([x])$.

This easily brings to a contradiction. Indeed, the above property implies that $\dim(Y)\le\dim(\pi_1^{-1}([x]))$ for $[x]\in\bP(R^1)$ general, and since
$$\dim(\Theta)=\dim(Y)+\dim(F_y)=\dim(\pi_1^{-1}([x]))+n$$

we also get $\dim(Y)\le\dim(Y)+\dim(F_y)-n$, which is impossible by Lemma \ref{LEM:easyresults}(c).
\end{proof}



\section{A lifting criterion for weak Lefschetz property}
\label{SEC:WLP}

It is known that, the $SLP$ for a graded algebra is inherited to its quotients by suitable conductor ideals (see, for example, \cite[Proposition 3.11]{bookLef}).
In this section we prove a sort of converse for $WLP_2$ for complete intersection SAGAs presented by quadrics. More precisely, we will give a criterion to reduce the proof of $WLP_2$ for a SAGA $R$ as above to a suitable quotient of $R$, modulo the existence of a non-Lefschetz element. We stress, moreover, that this criterion works for any codimension. 
\medskip

Let us start by setting $S=\bK[x_0,\dots, x_n]$ and by proving the following result.

\begin{lemma}
\label{LEM:quotsaga}
Assume that $R=S/J$ is a complete intersection SAGA of codimension $n+1$ presented by quadrics so the socle is in degree $N=n+1$. Assume that there are $w,z\in R^1\setminus\{0\}$ such that $zw=0$. Then $(z)=(0:w)$ and $\bar{R}=R/(z)$ is a complete intersection SAGA of codimension $n$ presented by quadrics. In particular, $\dim(K^s_w)=\dim(R^{s-1})-\dim(K^{s-1}_z)$.
\end{lemma}

\begin{proof}
By definition, we have $(z)\subseteq (0:w)$, so we can define an epimorphism of graded $\bK$-algebras $$\varphi:\bar{R}:=R/(z)\rightarrow \tilde{R}:=R/(0:w).$$
By \cite[Lemma 2.3]{FP21}, the latter is a SAGA of codimension $n$ and socle in degree $\tilde{N}=n$.
Considering $\bar{R}$, it is clearly an Artinian standard algebra of codimension $n$. We want to show that $\bar{R}$ is also a complete intersection SAGA presented by quadrics. 
By hypothesis, we know that $zw\in J$, so we can complete $\{zw\}$ to a regular sequence of the form $\{g_0,\cdots,g_{n-1},zw\}$ spanning $J$. Notice that, by construction, $g_0,\cdots,g_{n-1}$ do not belong to the ideal $(z)$ and the reductions $\bar{g}_i$ of $g_i$ modulo $(z)$ are a regular sequence of quadrics in the polynomial ring $\bar{S}=S/(z)$.

Hence we have
$$\bar{R}=R/(z)=\frac{S/J}{(z)}\simeq\frac{\bar{S}}{(\bar{g}_0,\cdots,\bar{g}_{n-1})}$$
so $\bar{R}$ is a complete intersection SAGA presented by quadrics. In particular, it has socle in degree $\bar{N}=n=\tilde{N}$.

Since $\varphi$ is an epimorphism and preserve the degrees, the image of a generator $\bar{\sigma}$ of $\bar{R}^n$ is a non-zero multiple of the generator $\tilde{\sigma}$ of $\tilde{R}^n$. This also implies the injectivity of $\varphi$. Indeed, let us take a non-zero element $x\in\bar{R}^i$. There exists $y\in\bar{R}^{n-i}$ such that $xy=\bar{\sigma}$. Hence, we have 
$$\lambda\tilde{\sigma}=\varphi(\bar{\sigma})=\varphi(xy)=\varphi(x)\varphi(y),$$
and so we get that $\varphi(x)$ can not be zero and $\varphi$ is an isomorphism. In particular, for all $s$,
$$R^{s-1}\cdot z=(z)_s=(0:w)_s=K^s_w$$
then we clearly have
$$\dim(K^s_w)=\dim(R^{s-1}z)=\dim(R^{s-1})-\dim(K^{s-1}_z)$$
as claimed.
\end{proof}

We assume again that $R=S/I$ is a complete intersection SAGA presented by quadrics with codimension $n+1$ so that it has socle in degree $N=n+1$. We recall that $[z]\in \bP(R^1)$ is a non-Lefschetz element (in degree $1$), i.e. $z\not \in \cL_1$, if and only if there exists $[w]\in \bP(R^1)$ such that $zw=0$ in $R$ or, equivalently $K^1_{z}=\langle w\rangle$.

\begin{theorem}
\label{THM:induction}
Let $R=S/I$ be a complete intersection SAGA of codimension $n+1\geq 6$ presented by quadrics and assume that $z$ is a non-Lefschetz element for $R$. If $\bar{R}=R/(z)$ satisfies the $WLP_2$, then the same holds for $R$.
\end{theorem}

\begin{proof}
First of all, by Lemma \ref{LEM:quotsaga}, we have that $\bar{R}=R/(z)=R/(0:w)$ is a complete intersection SAGA presented by quadrics and  $K^1_z=\langle w\rangle$ . Let  $\mathrm{pr}_z$ be the projection $R\to \bar{R}$. 
\medskip

Assume, by contradiction, that $WLP_2$ holds for $\bar{R}$ but not for $R$. In particular, for all $x\in R^1$, the multiplication map $x\cdot:R^2\rightarrow R^3$ has non trivial kernel, i.e. $K^2_x\ne\{0\}$. Consider the incidence correspondence 
$$\Gamma=\{([x],[v])\in \bP(R^1)\times \bP(R^1)\,|\, xvz=0\}$$
with its projections $p_1$ and $p_2$ on the factors. 
\medskip

We claim that $p_1$ is surjective. Since $\Gamma$ is a closed subset, it is enough to show that for $[x]\in \bP(R^1)$ general there exists $[v]\in \bP(R^1)$ such that $xvz=0$. Let $x$ be a general element of $R^1$. As $K^2_x\neq 0$ we have that there exists $[q]\in \bP(R^2)$ such that $xq=0$ in $R$. Then we have also $\mathrm{pr}_z(xq)=\overline{x}\overline{q}=0$ in $\bar{R}$. Since $[x]$ is general in $\bP(R^1)$, then the same holds for $[\bar{x}]\in\bP(\bar{R}^1)$, so we get $\bar{q}=0$ in $\bar{R}^2$, as $WLP_2$ holds for $\bar{R}$ by assumption. Then, by Lemma \ref{LEM:quotsaga}, we have $q\in (0:w)_2=(z)_2=z\cdot R^1$ so there exists $[v]\in \bP(R^1)$ such that $0=xq=xvz$ as claimed.
\medskip

In analogy with what happens for the construction \ref{CONSTR:OURFRAMEWORK}, we have that there exists a unique irreducible component $\Theta$ of $\Gamma$ which dominates $\bP(R^1)$ via $\pi_1$, where we set $\pi_i$ to be the restriction of $p_i$ to $\Theta$ for $i=1,2$. We have that for $[x]\in \bP(R^1)$ general 
$$\pi_1^{-1}([x])=p_1^{-1}([x])=[x]\times \bP(K^1_{xz})$$
so the general fiber of $\pi_1$ has dimension at most $1$ by Proposition \ref{PROP:SameKer}.
\medskip

Let us now show that the general fiber of $\pi_1$ has dimension $1$. Consider $[x]\in \bP(R^1)$ general. Firstly, let us observe that $([x],[w])$ belongs to $p_1^{-1}([x])=\pi_1^{-1}([x])$ since $zw=0$. 
As shown above, there exists $[q]\in \bP(R^2)$ such that $xq=0$ and $q=zv$ for suitable $[v]\in \bP(R^1)$. Moreover $[v]\ne[w]$ since, otherwise, $[q]$ would be zero, hence $\pi_1^{-1}([x])=\langle[w],[v]\rangle$ as claimed.
\medskip

By considering the second projection $\pi_2$, we have that for $[v]$ general in $Y=\pi_2(\Theta)$, the fiber $\pi_2^{-1}([v])$ is such that
$$\pi_2^{-1}([v])\subseteq p_2^{-1}([v])=\bP(K^1_{vz})\times [v],$$
which has dimension at most $1$ by Proposition \ref{PROP:SameKer}.
Since $\pi_1$ is dominant, for $([x],[v])\in \Theta$ general we have
$$n+1=\dim(\bP(R^1))+\dim(\pi_1^{-1}([x]))=\dim(\Theta)=\dim(Y)+\dim(\pi_2^{-1}([v])).$$
Since $\dim(Y)\le n$ and $\dim(\pi_2^{-1}([v]))\le 1$, for $v$ general, the only possibility is to have $\dim(\pi_2^{-1}([v]))=1$ and  $Y=\bP(R^1)$.
\medskip

We will show now that having $Y=\bP(R^1)$ gives a contradiction. First of all, by reasoning as in Proposition \ref{PROP:KERCOKERDelux}, one can prove that
$$Y\subseteq\{[v]\in\bP(R^1) \ | \ v^2z=0\}.$$
Since $Y=\bP(R^1)$ and squares of elements of $R^1$ generates $R^2$ (as $R$ is standard), we have that $z\cdot R^2=0$. This is impossible by Gorenstein duality.
\end{proof}

In the statement of Theorem \ref{THM:induction} we require $n+1\leq 6$ since for codimension $5$ the $WLP_2$ has already been proved (in \cite{AR19} or in \cite{BFP} as consequence of $SLP$) and in even smaller codimension, it easily follows from $WLP_1$ that is known to hold. From this, we get the following consequence.

\begin{corollary}
Let $R$ be a complete intersection SAGA of codimension $6$ presented by quadrics (e.g. $R$ is the jacobian ring of a cubic fourfold). If $\cL_1$ is not the whole $\bP(R^1)$, $R$ satisfies $WLP$.
\end{corollary}

\section{Some results for higher codimensions}
\label{SEC:Higher}

In this section, our aim is to obtain some results concerning Lefschetz properties for complete intersection SAGAs presented by quadrics with codimension equal to $n+1\ge 4$ by using techniques and results developed in the previous sections. 
\medskip

In particular, our framework will be as follows. Let us consider the polynomial ring $S=\bK[x_0\cdots,x_n]$, an ideal $I$ of $S$, generated by a regular sequence of quadrics $(g_0,\cdots,g_n)$ and the corresponding complete intersection SAGA $R=S/I$, which has codimension $n+1$ and socle in degree $N=n+1$. 
We recall that we have denoted by $SLP_1(k)$ the strong Lefschetz property in degree $1$ at range $k$: a SAGA $R$ will satisfy the property $SLP_1(k)$ if the multiplication map $x^k\cdot: R^1\rightarrow R^{k+1}$ is injective for $[x]$ general in $\bP(R^1)$. We recall that we have defined the loci $\Gamma_{i,j}^{(a,b)}$ in Section \ref{SEC:Notations}. For brevity, we will write $\Gamma_k$ to mean $\Gamma_{k,1}^{(1,1)}$.

We will prove the following:

\begin{theorem}
\label{THM:BRICO}
Let $k\in \{2,3,4\}$ and assume that $R$ is a complete intersection SAGA of codimension $n+1$ presented by quadrics. If $n\geq k+1$ we have that $R$ satisfies $SLP_1(k)$.
\end{theorem}

We will show the above theorem by splitting up the proof in $3$ cases, which will be treated in the Propositions \ref{PROP:SLP1,2gen}, \ref{PROP:SLP1,3gen} and \ref{PROP:SLP1,4gen} respectively.

\begin{proposition}
\label{PROP:SLP1,2gen}
Property $SLP_1(2)$ holds for every $n\ge3$.
\end{proposition}
\begin{proof}
Let us assume by contradiction that the multiplication map $x^2\cdot:R^1\rightarrow R^3$ is not injective for $[x]\in\bP(R^1)$. Then we can consider the locus $\Gamma_2=\{([x],[y])\in\bP(R^1)\times\bP(R^1) \ | \ x^2y=0\}$, with the corresponding varieties $\Theta$, $Y$ and $F_y$ defined as we have usually done. By Proposition \ref{PROP:KERCOKERDelux}, we obtain that $Y\subseteq\cN_3$. Then, by Lemmas \ref{LEM:tangnihil} and \ref{LEM:easyresults} we have that $\dim(Y)=1$: hence $F_y$ must be a hypersurface, which is absurd by Proposition \ref{PROP:Fyhyp}.
\end{proof}

The following result has been already proved when $n=4$ in \cite{BFP}.

\begin{proposition}
\label{PROP:SLP1,3gen}
Property $SLP_1(3)$ holds for every $n\ge4$.
\end{proposition}

\begin{proof}
Let us assume by contradiction that the multiplication map $x^3\cdot:R^1\rightarrow R^4$ is not injective for $[x]\in\bP(R^1)$.
As in the proof of Proposition \ref{PROP:SLP1,2gen}, let us construct $\Gamma_3$, $\Theta$, $Y$ and $F_y$. In this case, we have that $Y\subseteq\cN_4$ so $\dim(Y)\le2$. If $[y]\in Y$ is general, the only possible value for $(\dim(Y),\dim(F_y))$ is $(2,n-2)$ since $F_y$ can not be a hypersurface by Proposition \ref{PROP:Fyhyp}. 
By dimension reasons, $Y\not\subseteq\cN_3$, thus we have that the general element $[y]$ of $Y$ is such that $y^3\ne0$. Then, by Proposition \ref{PROP:KERCOKERDelux}, we get $F_y$ must be contained in $\bP(K^1_{y^3})$, whose dimension is at most $2$. 
Then we have proved that $n-2=\dim(F_y)\leq 2$ which is impossible for $n\ge5$. 
\end{proof}

Before showing the analogous result for the $SLP_1(4)$, let us prove a technical result.
\medskip

\begin{lemma}
\label{LEM:tech1}
Let $R$ be a SAGA as at the beginning of Section \ref{SEC:Higher} and consider $4\leq k\leq n-1$. Assume that $R$ does not satisfy $SLP_1(k)$ so one can consider the varieties $\Gamma_k,\Theta,Y$ and $F_y$ constructed as in Section \ref{SEC:Notations}. For $[y]\in Y$ general we have the following properties.
\begin{enumerate}[(a)]
    \item If $R$ satisfies $SLP_{1}(k-1)$, then $\dim(F_y)\leq k-1$;
    \item $(\dim(Y),\dim(F_y))\ne(k-1,k-1)$.
\end{enumerate}
\end{lemma}
\begin{proof}
For $(a)$, let us assume by contradiction that for $[y]\in Y$ general, $\dim(F_y)=h\ge k$. Then, by Lemma \ref{LEM:xtangy}, we have that for $[x]\in F_y$ general $$T_{[x]}(F_y)=\bP(T_{\tilde{F}_y,x})\subseteq \bP(K^1_{x^{\alpha}y^{\beta}}),$$
where $\alpha+\beta=k$, with $\beta\ge1$ and $\tilde{F}_y$ is the affine cone over $F_y$. But since $SLP_1(k-1)$ holds for $R$ by hypothesis, for $([x],[y])\in \Theta$ general, we have that $x^{k-1}y\ne0$, and so 
$$h=\dim(F_y)\le\dim(\bP(K^1_{x^{k-1}y}))\le k-1,$$
which is impossible by the assumptions over $h$.

\medskip

For $(b)$, let us consider $[y]\in Y$ general and assume by contradiction that $\dim(Y)=\dim(F_y)=k-1$. By Proposition \ref{PROP:KERCOKERDelux} we get that $Y\subseteq\cN_{k+1}$, and by Lemma \ref{LEM:tangnihil} we deduce that $Y$ is an irreducible component of $\cN_{k+1}$ and we have that $y^k\ne0$. By reasoning as in the proof of Proposition \ref{PROP:SLP1,3gen} and by Proposition \ref{PROP:SameKer}, we get $F_y=\bP(K^1_{y^k})$.
Moreover, since $[y]$ is general in $Y$, we have $T_{[y]}(Y)=\bP(K^1_{y^k})=F_y$. With these conditions, we can proceed as in the proof of Proposition \ref{PROP:2,3} and consider for example the equations
$$0\equiv(x+tv)^{k-1}(y+tw)^2\, (\mbox{mod }\, t^2)\qquad 0\equiv (x+tv)^k(y+tw)\, (\mbox{mod }\, t^2),$$
where $w\in T_{\tilde{Y},y}$ and $(v,w)\in T_{\tilde{\Theta},(x,y)}$. In this way, we get that $Y\subseteq\pi_1^{-1}([x])$, where, as usual, $\pi_1: \Theta\subseteq \bP(R^1)\times \bP(R^1)\rightarrow\bP(R^1)$ is the first projection. This leads to a contradiction as shown in Proposition \ref{PROP:2,3}.
\end{proof}

We can now show the last case we need to prove Theorem \ref{THM:BRICO}. 

\begin{proposition}
\label{PROP:SLP1,4gen}
Property $SLP_1(4)$ holds for every $n\ge5$.
\end{proposition}

\begin{proof}
First of all, let us notice that the statement has been already proved for $n=5$ in Section \ref{SEC:SLP1}: we have to prove $SLP_1(4)$ for $n\ge6$.

\medskip

Let us assume that for $x\in R^1$, the multiplication map $x^4\cdot:R^1\rightarrow R^5$ is not injective. As usual, we can then consider the incidence correspondence $\Gamma_4=\{([x],[y])\in\bP(R^1)\times\bP(R^1) \ | \ x^4y=0\}$ and the corresponding varieties $\Theta$, $Y$ and $F_y$, for $[y]\in Y$ general. \\
By Proposition \ref{PROP:KERCOKERDelux}, we get that $Y\subseteq\cN_5$, so $\dim(Y)\le3$. By using the bounds of Lemma \ref{LEM:easyresults}(a) and Proposition \ref{PROP:Fyhyp}, the only possible cases for the values of $(\dim(Y),\dim(F_y))$ are
$$(2,n-2) \qquad (3,n-2) \qquad (3,n-3).$$
By Proposition \ref{PROP:SLP1,3gen}, we know that $SLP_1(3)$ holds for $n\ge6$. Then, by Lemma \ref{LEM:tech1}$(a)$, we get that $\dim(F_y)$ is at most $3$: the cases $(\dim(Y),\dim(F_y))=(2,n-2)$ and $(\dim(Y),\dim(F_y))=(3,n-2)$ can not occur for every $n\ge6$. We also have that $(\dim(Y),\dim(F_y))\ne(3,n-3)$ for every $n\ge7$. 

\medskip

The only case we have still to analyze is the one with $n=6$ and $\dim(Y)=\dim(F_y)=3$. By considering Lemma \ref{LEM:tech1}$(b)$, we can rule out this last possibility too: $SLP_1(4)$ holds for $R$, for every $n\ge5$.
\end{proof}

We conclude this section by observing how much can be easily said, by using these methods, for the $SLP_1$ for complete intersection SAGA of codimension $7$ (i.e. $n=6$) presented by quadrics (e.g. for the Jacobian ring of a smooth cubic fivefold).

\begin{corollary}
Let  $R$ be a complete intersection SAGA of codimension $7$ presented by quadrics. We have that $SLP_1(4)$ holds. Moreover, if $SLP_1(5)$ does not hold (i.e. if $SLP_1$ does not hold), then one can construct the varieties $\Gamma_5,\Theta,Y$ and $F_y$ as usual and we have, for $[y]\in Y$ general, $(\dim(Y),\dim(F_y))\in \{(2,4),(3,3)\}$.
\end{corollary}

\begin{proof}
Property $SLP_1(4)$ holds by Theorem \ref{THM:BRICO}. Assume that $SLP_1(5)$ does not hold for $R$. Then. by Lemma \ref{LEM:easyresults}, we have 
\begin{equation}
\label{EQ:bounds2}
1\le\dim(Y)\le4 \qquad 2\le\dim(F_y)\le5 \qquad \dim(Y)+\dim(F_y)=\dim(\Theta)\ge 6.
\end{equation}
As in Theorem \ref{THM:SLP1}, we put in a table the possible values of the pairs $(\dim(Y),\dim(F_y))$ and we specify which result rules out the corresponding case.

\begin{center}
\begin{tabular}{c||c|c|c|c|}
$\dim(Y)$ VS $\dim(F_y)$ & $2$ & $3$ & $4$ & $5$\\
\hline\hline
1 & \eqref{EQ:bounds2} & \eqref{EQ:bounds2} & \eqref{EQ:bounds2} & Prop. \ref{PROP:Fyhyp}\\ \hline
2 & \eqref{EQ:bounds2} & \eqref{EQ:bounds2} &  & Prop. \ref{PROP:Fyhyp}\\ \hline
3 & \eqref{EQ:bounds2} &  & Prop. \ref{PROP:2,3} + Rem. \ref{REM:dimnihil} & Prop. \ref{PROP:Fyhyp}\\
\hline
4 & Prop. \ref{PROP:dimYdiversoda_nmeno2} & Prop. \ref{PROP:dimYdiversoda_nmeno2} & Prop. \ref{PROP:dimYdiversoda_nmeno2} & Prop. \ref{PROP:Fyhyp} \\ \hline
\end{tabular}
\end{center}
This concludes the proof.
\end{proof}


\end{document}